\documentclass[reqno, 11pt]{amsart}
\usepackage{graphicx, enumerate, enumitem, amsmath, amssymb,amsthm,mathabx,manfnt,hyperref,braket,amscd,
mathrsfs, mathtools, empheq, verbatim, placeins, floatrow}
\usepackage[utf8]{inputenc}
\usepackage[T1]{fontenc}
\numberwithin{equation}{section}
\newtheorem{theorem}{Theorem}
\newtheorem{proposition}{Proposition}[section]
\newtheorem{lemma}[proposition]{Lemma}

\theoremstyle{definition}
\newtheorem{definition}{Definition}
\newcommand{\bs}[1]{{\boldsymbol{#1}}}
\newcommand{\rl}{\mathbb{R}}

\newcommand{\sx}{\mathbb{S}}

\newcommand{\abs}[1]{\left\vert#1\right\vert}
\newcommand{\ba}{\mathbf{A}}
\newcommand{\bb}{\mathbf{B}}
\title{Curves of constant curvature and torsion in the 3-sphere}
\author{Debraj Chakrabarti}
\address{Department of Mathematics, Central Michigan University, Mt. Pleasant,  MI 48859, USA}
\email{chakr2d@cmich.edu}
\urladdr{http://people.cst.cmich.edu/chakr2d/}
\author{Rahul Sahay}
\address{Central Michigan University, Mt. Pleasant,  MI 48859,  USA}
\email{sahay1r@cmich.edu}
\author{Jared Williams}
\address{Central Michigan University, Mt. Pleasant,  MI 48859,  USA}
\email{willi9jr@cmich.edu}
\thanks{\noindent All three authors were partially supported by a grant from the NSF (\#1600371). Debraj Chakrabarti was partially supported by a grant from the Simons Foundation (\# 316632) and also by an 
Early Career internal grant from Central Michigan University.}
\begin{document}
\maketitle
\begin{abstract}
    We describe the curves of constant (geodesic) curvature and torsion in the three-dimensional round sphere. These curves 
    are the trajectory of a point whose motion is the superposition of two circular motions in orthogonal planes. The global 
    behavior may be periodic or the curve may be dense in a Clifford torus embedded in the three-sphere. This behavior is 
    very different from that of helices in three-dimensional Euclidean space, which also have constant curvature and torsion.
\end{abstract}
\section{Introduction} \label{sec-introduction}
Let $\left(M,\langle,\rangle\right)$ be a three dimensional Riemannian manifold, let $I\subseteq\rl$ be an open interval, and let $\bs{\gamma}:I\to M$ be a smooth curve in $M$, which we assume to be parametrized by the arc length $t$. It is well-known that the  local geometry of $\bs{\gamma}$ is characterized by the {\em curvature} $\kappa$ and the 
{\em torsion} $\tau$. These are functions defined 
along $\bs{\gamma}$ and are the coefficients of the well-known {\em Frenet-Serret formulas} \cite[Vol. IV, pp. 21-23]{spivak}:
\begin{alignat}{4} 
    \frac{D}{dt}{\mathbf{T}(t)} &= &&\phantom{-}\kappa(t) \mathbf{N}(t)&&&                                            \nonumber\\
    \frac{D}{dt}{\mathbf{N}(t)} &= -\kappa(t){\mathbf{T}}(t) && &&& +\tau(t) \mathbf{B}(t)\label{eq-frenet}\\
    \frac{D}{dt}{\mathbf{B}(t)} &= &&-\tau(t) {\mathbf{N}}(t),&&& \nonumber
\end{alignat}
where the orthogonal unit vector fields $\mathbf{T,N,B}$, with $\mathbf{T} = \bs{\gamma}'$, along the unit-speed curve $\bs{\gamma}$, constitute its {\em Frenet frame} and $\frac{D}{dt}$ denotes covariant differentiation along $\bs{\gamma}$ with respect to the arc length $t$. We will assume that each of the functions $\kappa$ and $\tau$ is either nowhere zero or vanishes identically. Additionally, if $\kappa$ is identically zero, then $\tau$ is also taken to be identically zero. For completeness,
we include a proof of the set of Equations given in~\eqref{eq-frenet} in Section~\ref{sec-3Dmanifold} below. We make the following definition:

\begin{definition} \label{helixDef}
Let $M$ be a Riemannian manifold of dimension 3. A curve\\$\bs{\gamma}:I\to M$, where $I\subseteq\rl$ is an open interval, will be called a {\em helix} (plural: {\em helices}) if its curvature $\kappa$ and torsion $\tau$ are non-negative constants. A helix is {\em non-degenerate} if $\kappa$ and $\tau$ are both positive, and {\em degenerate} otherwise. We say that the helix $\bs{\gamma}$ is {\em periodic} if there is a $T>0$ such that $\bs{\gamma}(t+T)= \bs{\gamma}(t)$ for each $t \in I$.
\end{definition}

We take $\tau$ to be non-negative since we use the non-oriented form of the Frenet-Serret Equations (see Section~\ref{sec-frenet}). Definition~\ref{helixDef} is motivated by the example of the Euclidean space $\rl^3$, where non-degenerate helices are curves of the form:
\begin{align}
 \bs{\gamma}(t)=\cos(\omega t)\mathbf{A}+\sin(\omega t)\mathbf{B} +t\mathbf{C}+\mathbf{D}\label{eq-euclidean}
 \end{align}
where $\mathbf{A},\mathbf{B},\mathbf{C}, \mathbf{D}\in \rl^3$, $\mathbf{A},\mathbf{B},\mathbf{C}$ are non-zero and orthogonal  with $\abs{\mathbf{A}}=\abs{\mathbf{B}}$, and $\omega>0$.  These are elegant curves that are invariant under a one-parameter group of isometries of the ambient space. Note that there are no non-degenerate periodic helices in $\rl^3$.

The aim of this paper is to study helices in the three dimensional round sphere $\sx^3$. Thanks to the fact that $\sx^3$ is compact, we expect that a non-degenerate helix in $\sx^3$ should ``come back where it started from'' provided we wait long enough, and therefore, there is a possibility that, for favorable choice of  the curvature and torsion, the helix is actually periodic, though locally it is not much different from a helix in $\rl^3$. Globally, a non-degenerate helix in $\sx^3$ has {\em two} fundamental angular frequencies, $\omega_1$ and $\omega_2$, as opposed to the single fundamental angular frequency, $\omega$, of the helix given by Equation~\eqref{eq-euclidean} in $\rl^3$. A non-degenerate helix in $\sx^3$ may be thought of as the trajectory of a particle which performs two superimposed circular motions with frequencies $\omega_1$ and $\omega_2$. These fundamental angular frequencies must satisfy
    \[ \omega_2 > 1 > \omega_1, \]
a constraint which arises because a curve with non-zero curvature and torsion must lie in the positively curved compact space $\sx^3$. Unlike in $\rl^3$ where non-degenerate helices are embedded, non-compact submanifolds, depending on the fundamental angular frequencies $\omega_1$ and $\omega_2$, a non-degenerate helix in $\sx^3$ can either be periodic (when it is a compact embedded submanifold) or be dense in a flat 2-torus contained in $\sx^3$ (when the image of the helix is not an embedded submanifold of $\sx^3$). This divergence of global behavior from the flat case is the main topic of this paper.

Of course, the same questions can be asked in any number of spatial dimensions, and for other Riemannian manifolds. Here, for simplicity we consider the special case of $\sx^3$, which also allows us to use only elementary calculus-based methods in our investigations. Our methods will likely generalize to round spheres of any number of dimensions.

\section{Main results} \label{sec-mainresults}
We consider $\sx^3$ to be embedded in the Euclidean space $\rl^4$ in the natural way as the hypersurface $\{ x_1^2+x_2^2+x_3^2+x_4^2=1\}$, and endow $\sx^3$ with the Riemannian metric induced from $\rl^4$. This entails no loss of generality because the metric so induced is the same as the standard round metric of $\sx^3$ with constant sectional curvature $1$. To state our results concisely, let us introduce the following definition:

\begin{definition} \label{def-lissajous}
A smooth curve $\bs{\gamma}$ in $\rl^4$ will be called a {\em Lissajous curve} if there are  numbers $ \omega_2> \omega_1\geq0$ and vectors $\ba_1, \bb_1, \ba_2, \bb_2\in \rl^4$ such that, for each $t$, 
\begin{align}\label{eq-lissajous}
    \bs{\gamma}(t)=\cos(\omega_1 t)\ba_1+ \sin(\omega_1 t)\bb_1+\cos(\omega_2 t)\ba_2+\sin(\omega_2 t)\bb_2.
\end{align}
We will call $\omega_1$ and $\omega_2$ the {\em fundamental angular frequencies} of the curve $\bs{\gamma}$ and $\ba_1, \bb_1, \ba_2, \bb_2$ the {\em coefficient vectors} of $\bs{\gamma}$.
\end{definition}

Therefore, a Lissajous curve, in our sense, can be thought of as the trajectory of a point in $\rl^4$ which oscillates with frequency $\omega_1$ in the $\ba_1 \bb_1$-plane and with frequency $\omega_2$ in the $\ba_2\bb_2$-plane. Note also that the projection of $\bs{\gamma}$ on any 2-dimensional linear subspace different from the $\ba_1\bb_1$ and $\ba_2\bb_2$-planes is a planar Lissajous curve in the usual sense of the term \cite[pp. 114-115]{katok}. We are now ready to describe helices in $\sx^3$:

\begin{theorem}\label{thm-main} 
Let $\kappa, \tau \geq 0$ be given numbers where, if $\kappa = 0$, then $\tau = 0$. 
\begin{enumerate}[leftmargin=*] \itemsep5pt
    \item There exists a helix $\bs{\gamma}:\rl \to\sx^3$ with constant curvature $\kappa$ and torsion $\tau$.
    
    \item Such a helix $\bs{\gamma}$ is a Lissajous curve in the form of Equation \eqref{eq-lissajous}. 
    
    \item The fundamental angular frequencies of $\bs{\gamma}$ are distinct and are given by
    \begin{align}
        \omega_1 & = \sqrt{\frac{\chi^2 - \sqrt{\chi^4 - 4 \tau^2}}{2}}\label{eq-omega1} \\
        \omega_2 & = \sqrt{\frac{\chi^2 + \sqrt{\chi^4 - 4 \tau^2}}{2}}\label{eq-omega2}
    \end{align}
    with 
    \begin{align}
        \chi^2 = \kappa^2 + \tau^2 + 1. \label{eq-chidef}
    \end{align}

    \item If $\kappa > 0$, then the frequencies $\omega_1$ and $\omega_2$ satisfy
        \begin{align}\label{eq-omegaconstraints}
            \omega_2 > 1 > \omega_1.
        \end{align}
   
    \item If $\tau \neq 0$ then the four coefficient vectors $\ba_1, \bb_1, \ba_2,\bb_2$ are orthogonal in $\rl^4$, and their magnitudes are given by
        \begin{align}\label{eq-a1b1}
            \abs{\ba_1}^2 = \abs{\bb_1}^2 = \frac{1-\omega_2^2}{\omega_1^2-\omega_2^2}
        \end{align}
        and
        \begin{align}\label{eq-a2b2}
            \abs{\ba_2}^2 = \abs{\bb_2}^2 = \frac{1-\omega_1^2}{\omega_2^2-\omega_1^2}.
        \end{align}
        
    \item If $\tau = 0$, then $\bs{\gamma}$ is a circle given by 
    $$\bs{\gamma}(t) = \ba_1 + \cos(\omega t) \ba_2 + \sin(\omega t) \bb_2$$
    where $\omega = \sqrt{\kappa^2 + 1}$. Further, the coefficient vectors $\ba_1, \ba_2, \bb_2$ are mutually orthogonal and we have
        \begin{align*}
            \abs{\ba_2} &= \abs{\bb_2} = \frac{1}{\omega} \quad \text{ and } \quad \abs{\ba_1} = \sqrt{1 - \frac{1}{\omega^2}}.
       \end{align*}
\end{enumerate}
\end{theorem}
Several interesting features may be noted here: 
\begin{enumerate}[leftmargin=*] \itemsep5pt
    \item The local existence of helices follows from the existence theorem for             solutions of systems of ordinary differential equations on manifolds. However,      we prove the existence of helices directly by solving the Frenet-Serret             equations and obtain an explicit representation of the solution.

    \item When $\kappa, \tau > 0$, the curve $\bs{\gamma}$ may be thought of as the trajectory of a motion consisting of two superimposed circular motions in perpendicular planes: one at a ``slow'' frequency $\omega_1 < 1$ and the other at a 
    ``fast'' frequency $\omega_2> 1$. This global behavior arises from the fact that the curve $\bs{\gamma}$ must lie on the compact surface $\sx^3$. Observe that there is no such restriction on the angular frequency $\omega$ of the Euclidean helix given by Equation \eqref{eq-euclidean}.
    
    \item When $\kappa=0$ by definition we have $\tau = 0$. Such a curve is a {\em geodesic}, i.e., its unit tangent field is auto-parallel along the curve. Therefore, geodesics on the sphere $\sx^3$ are of the form
    \[\bs{\gamma}(t) = \cos(t) \ba + \sin(t)\bb \]
    where $\abs{\ba} = \abs{\bb} = 1$ and $\ba, \bb$ are mutually orthogonal. Thus, we recapture the well known fact that geodesics in $\sx^3$ are great circles.
\end{enumerate}

We now turn to the question of uniqueness and periodicity of helices. First, note that if $\bs{\gamma}$ is a helix in a Riemannian 3-manifold, $M$, and $f:M\to M$ is a self-isometry of $M$, then $f\circ \bs{\gamma}$ is also a helix in $M$ with the same curvature and torsion as that of $\bs{\gamma}$. In $\rl^3$, the converse holds, i.e., helices with the same curvature and torsion are congruent under an isometry of $\rl^3$. We show that a similar fact holds in $\sx^3$ and we also determine when helices are periodic.

Recall that a {\em Clifford torus} is a Riemannian 2-manifold which is the metric product of two circles. Clearly, a Clifford torus is flat, i.e., its Gaussian curvature vanishes identically. It is well-known that there are Clifford tori embedded in the sphere $\sx^3$, e.g. for $0<\lambda<1$, the surface in $\rl^4$ given by
\[ \mathcal{C}_{\lambda}=\left\{ \mathbf{x}\in\rl^4 : x_1^2+ x_2^2=\lambda, x_3^2+x_4^2=1 - \lambda\right\}\]
is clearly contained in $\sx^3$, and is therefore a Clifford torus in $\sx^3$ which is flat in the Riemannian metric induced by the round metric of $\sx^3$.

\begin{theorem}\label{thm-existence} Let $\kappa, \tau \geq 0$ be given numbers where, if $\kappa = 0$, then $\tau = 0$. 
\begin{enumerate}[leftmargin=*] \itemsep5pt
    \item If $\bs{\alpha}$ and $\bs{\beta}$ are two helices in $\sx^3$ with the same curvature $\kappa$ and torsion $\tau$, then $\bs{\alpha}$ and $\bs{\beta}$ are congruent, i.e., there is a Riemannian isometry $f:\sx^3\to \sx^3$ such that $\bs{\beta}= f \circ \bs{\alpha}$.
    
    \item A helix $\bs{\gamma}$ is periodic if and only if the ratio of the angular frequencies 
    \begin{align}
        \cfrac{\omega_1}{\omega_2} = \sqrt{\frac{\chi^2 - \sqrt{\chi^4 - 4 \tau^2}}{\chi^2 + \sqrt{\chi^4 - 4 \tau^2}}}
    \end{align}
    is a rational number, where $\chi^2 = \kappa^2 + \tau^2 + 1$.
    
    \item If $\kappa, \tau > 0$, there exists a Clifford torus $\mathbb{T}^2_{\bs{\gamma}}$ contained in $\sx^3$ such that the image of $\bs{\gamma}$ lies on $\mathbb{T}^2_{\bs{\gamma}}$.
    
    \item If $\kappa, \tau > 0$ and $\cfrac{\omega_1}{\omega_2}$ is irrational, the image of $\bs{\gamma}$ is dense in the torus $\mathbb{T}_{\bs{\gamma}}^2$.
\end{enumerate}
\end{theorem}

\FloatBarrier
\subsection{Visualization of Helices}
One way to visualize $\sx^3$ is to use the sterographic projection $\sigma:\sx^3\setminus\{p\} \to \rl^3$ where $p$ is a point in $\sx^3$ which serves as the pole of the projection. It is well-known that $\sigma$ is conformal, i.e. it preserves angles but not lengths. Using Wolfram Mathematica\texttrademark, we produced visualizations of two helices in $\sx^3$ which are shown in Figures~\ref{fig-dense}~and~\ref{fig-period} below. Each of these pictures represents two distinct perspective projections onto $\rl^2$ of the sterographic projection of the helix, where $p$ is chosen to not be on the helix. The helix in Figure~\ref{fig-dense} is non-periodic and therefore dense in a Clifford torus. The helix in Figure~\ref{fig-period} is periodic and therefore an embedded curve in $\sx^3$. The hue and brightness of the following curves are functions of the fourth coordinate of the curve $\bs{\gamma}$ in its embedding in $\rl^4$. 
\begin{figure}[H]
    \centering
    \includegraphics[width=0.5\textwidth]{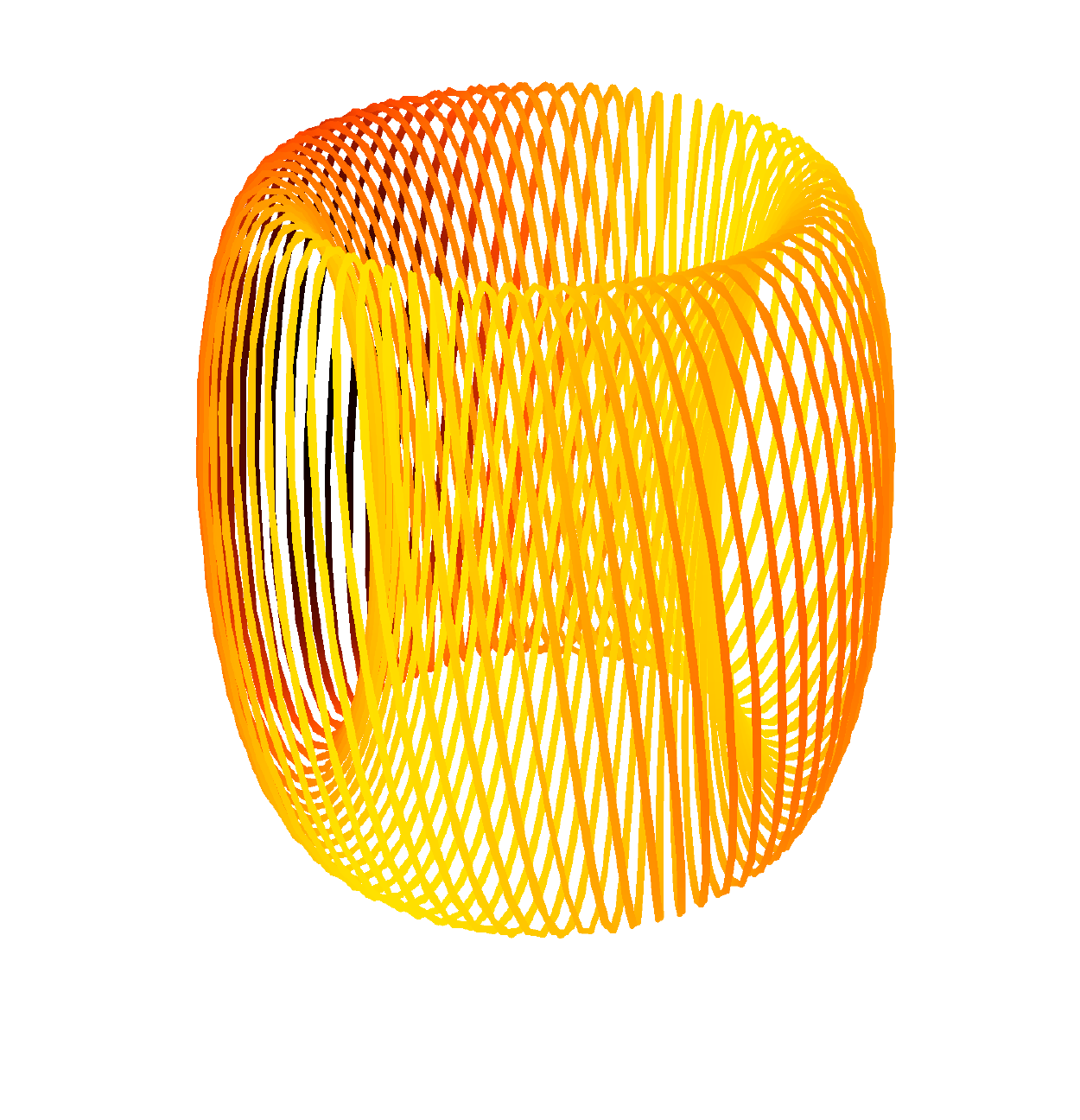}\includegraphics[width =0.5\textwidth] {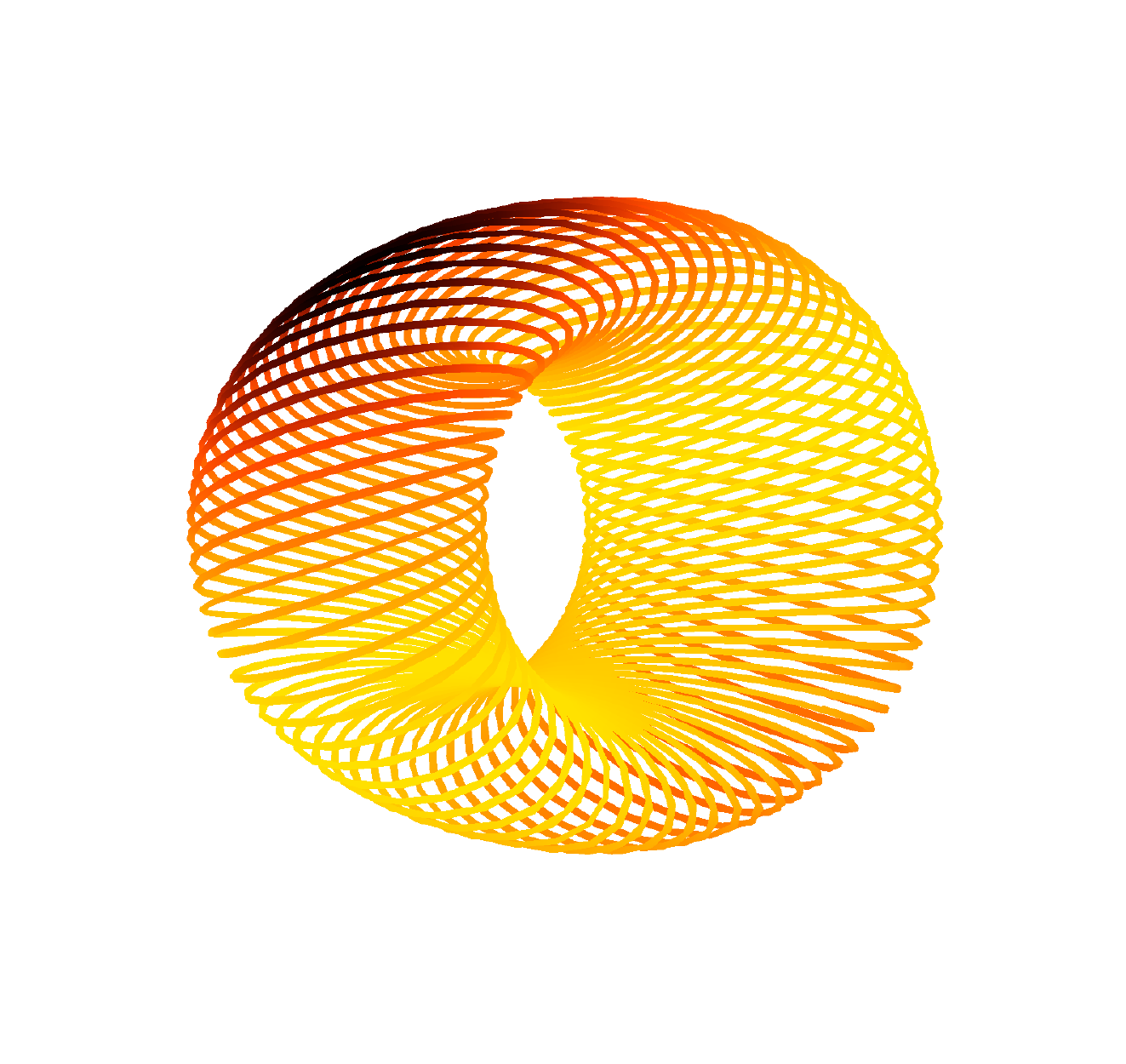}
    \caption{Two views of a dense helix in $\sx^3$ with $\kappa = \frac{5\sqrt{3}}{4}$ and $\tau = \frac{\sqrt{29}}{4}$. The corresponding fundamental angular frequencies are then $\omega_1=\frac{1}{2}, \omega_2=\frac{\sqrt{29}}{2},$ and thus, their ratio is the irrational number $\frac{\omega_2}{\omega_1} = \sqrt{29}$.}
    \label{fig-dense}
\end{figure}
\begin{figure}[H]
    \centering
    \includegraphics[width=0.5\textwidth]{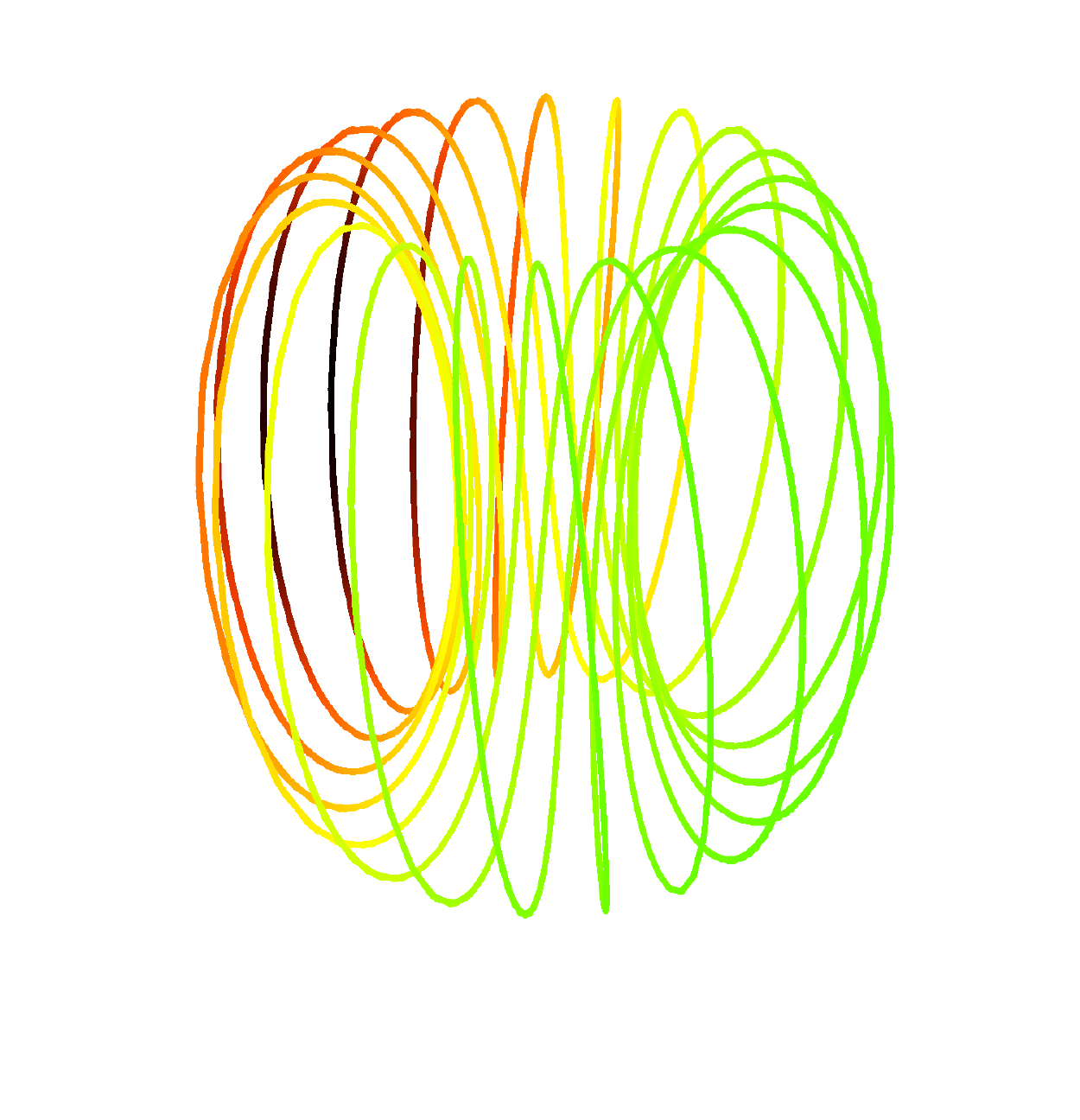}\includegraphics[width=0.5\textwidth]{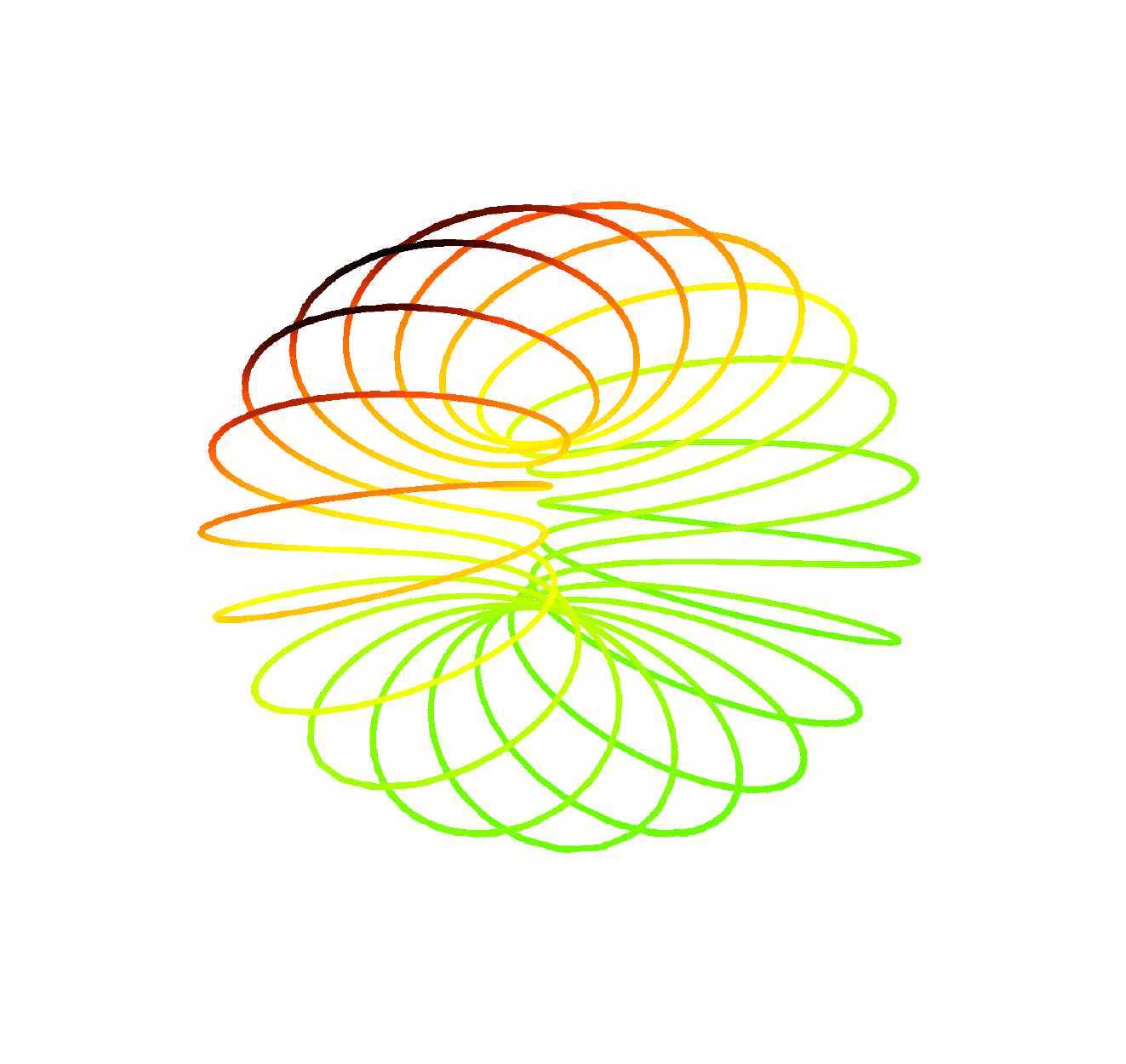}
    \caption{Two views of a periodic helix in $\sx^3$ with $\kappa =\frac{\sqrt{15}}{3}$ and $\tau = \frac{5}{12}$. The corresponding fundamental angular frequencies are then $\omega_1=\frac{1}{4}, \omega_2=\frac{5}{3},$ and thus, their ratio is the rational number $\frac{\omega_2}{\omega_1} = \frac{20}{3}$.}
    \label{fig-period}
\end{figure}
\FloatBarrier

\newpage
\section{The Frenet-Serret Equations} \label{sec-frenet}
\subsection{The Frenet-Serret Equations in a 3-Dimensional Riemannian Manifold} \label{sec-3Dmanifold}
Consider a three dimensional Riemannian manifold $({M}, \langle,\rangle)$ and an arc length parameterized curve $\boldsymbol{\gamma}: I \to M$ where $I\subset \rl$ is an open interval. Let $\frac{D}{dt}$ represent the covariant derivative of a vector field along a curve (parametrized by $t$).  Let $\mathbf{T} = \boldsymbol{\gamma}'$ denote the {\em unit tangent vector field} of $\boldsymbol{\gamma}$.  Since $\left\langle \mathbf{T}(t), \mathbf{T}(t) \right\rangle = 1$ for each $t$ we have,
\[
    0 = \frac{d}{dt} \left\langle \mathbf{T}(t), \mathbf{T}(t) \right\rangle = 2 \left\langle \mathbf{T}(t), \frac{D}{dt} \mathbf{T}(t) \right\rangle.
\]
Then, the curvature function $\kappa$ is defined as
\begin{align} \label{eq-kappa}
    \kappa(t) = \left| \frac{D}{dt} \mathbf{T}(t) \right|.
\end{align}
We will assume that either $\kappa(t) \neq 0$ for all $t$ or that $\kappa \equiv 0$. In the case where $\kappa$ never vanishes, we define the {\em normal vector field} to $\boldsymbol{\gamma}$ by
\begin{align*}
    \mathbf{N}(t) = \frac{1}{\kappa(t)} \frac{D}{dt} \mathbf{T}(t).
\end{align*}
Then $\mathbf{N}$ is a unit vector field along $\boldsymbol{\gamma}$ which is always orthogonal to $\mathbf{T}$.  The definition of $\mathbf{N}$ gives the first Frenet-Serret equation
\begin{align} \label{eq-F1}
    \frac{D}{dt} \mathbf{T}(t) = \kappa(t) \mathbf{N}(t).
\end{align}
Similarly, since $\left\langle \mathbf{N}(t), \mathbf{N}(t) \right\rangle = 1$ for each $t \in I$
\begin{align*}
    2 \left\langle \mathbf{N}(t), \frac{D}{dt} \mathbf{N}(t) \right\rangle &= 0,
\end{align*}
and  because $\left\langle \mathbf{T}(t), \mathbf{N}(t) \right\rangle = 0$ for each $t \in I$
\begin{align*}
    \left\langle \frac{D}{dt} \mathbf{T}(t), \mathbf{N}(t) \right\rangle + \left\langle \mathbf{T}(t), \frac{D}{dt} \mathbf{N}(t) \right\rangle &= \kappa(t) + \left\langle \mathbf{T}(t), \frac{D}{dt} \mathbf{N}(t) \right\rangle = 0
\end{align*}
by Equation \eqref{eq-F1}.  Then,
\begin{align*}
    \frac{D}{dt} \mathbf{N}(t) = -\kappa(t) \mathbf{T}(t) + {\text{vector orthogonal to } \mathbf{T}(t) \text{ and } \mathbf{N}(t)}.
\end{align*}
We define a torsion function $\tau$ by
\begin{align}
    \tau(t) = \left| \frac{D}{dt} \mathbf{N}(t) + \kappa(t) \mathbf{T}(t). \right|\label{eq-tau}
\end{align}
We will assume that either $\tau \neq 0$ for all $t$, or $\tau \equiv 0$. If $\tau(t) \neq 0$ for all $t$, then we set
\begin{align*}
    \mathbf{B}(t) = \frac{1}{\tau(t)} \left( \frac{D}{dt} \mathbf{N}(t) + \kappa(t) \mathbf{T}(t) \right)
\end{align*}
such that $\mathbf{B}$ is a unit vector field along $\bs{\gamma}$ which is orthogonal to $\mathbf{T} \text{ and } \mathbf{N}$ for all $t$. If $\tau \equiv 0$, then we choose $\mathbf{B}$ to be an auto-parallel vector field along $\bs{\gamma}$ such that the vectors $\mathbf{T}(t), \mathbf{N}(t) \text{ and } \mathbf{B}(t)$ form an orthonormal basis of $T_{\bs{\gamma}(t)} \sx^3$.
In both cases we have
\begin{align} \label{eq-F2}
    \frac{D}{dt} \mathbf{N}(t) = - \kappa(t) \mathbf{T}(t) + \tau(t) \mathbf{B}(t).
\end{align}
Finally, since $\left\langle \mathbf{B}(t), \mathbf{B}(t) \right\rangle = 1$ for each $t \in I$
\begin{align*}
    \left\langle \mathbf{B}(t), \frac{D}{dt} \mathbf{B}(t) \right\rangle = 0,
\end{align*}
and because  $\left\langle \mathbf{N}(t), \mathbf{B}(t) \right\rangle = 0$ for each $t \in I$
\begin{align*}
    \left\langle \frac{D}{dt} \mathbf{N}(t), \mathbf{B}(t) \right\rangle + \left\langle \mathbf{N}(t), \frac{D}{dt} \mathbf{B}(t) \right\rangle &= \tau(t) + \left\langle \mathbf{N}(t), \frac{D}{dt} \mathbf{B}(t) \right\rangle
\end{align*}
by Equation \eqref{eq-F2}.  Then,
\begin{align*}
    \frac{D}{dt} \mathbf{B}(t) &= -\tau(t) \mathbf{N}(t) + {\text{vector orthogonal to } \mathbf{N}(t) \text{ and } \mathbf{B}(t)}.\\
    &\implies \frac{D}{dt} \mathbf{B}(t) = -\tau(t) \mathbf{N}(t) + c \mathbf{T}(t)
\end{align*}
since $\mathbf{T}(t)$ is orthogonal to $\mathbf{N}(t) \text{ and } \mathbf{B}(t)$ for each $t \in I$ by construction.  By taking the dot product of both sides with $\mathbf{T}(t)$ we have
\begin{align*}
    \left\langle \frac{D}{dt} \mathbf{B}(t), \mathbf{T}(t) \right\rangle = -\tau(t) \mathbf{N}(t) + c \mathbf{T}(t) = c
\end{align*}
and by the product rule
\begin{align*}
    \left\langle \frac{D}{dt} \mathbf{B}(t), \mathbf{T}(t) \right\rangle &= \frac{d}{dt} \left\langle \mathbf{B}(t), \mathbf{T}(t) \right\rangle - \left\langle \mathbf{B}(t), \frac{D}{dt} \mathbf{T}(t) \right\rangle = 0\\
    &\implies c = 0.
\end{align*}
Therefore we have our third and final Frenet-Serret equation
\begin{align} \label{eq-F3}
    \frac{D}{dt} \mathbf{B}(t) = -\tau(t) \mathbf{N}(t)
\end{align}
Equations \eqref{eq-F1}, \eqref{eq-F2} and \eqref{eq-F3} constitute the {\em Frenet-Serret equations} in a 3-dimensional Riemannian manifold and characterize the local geometry of the curve $\bs{\gamma}$. This concludes the derivation of the Frenet-Serret formulas in the case where $\kappa(t)\neq 0$ for all $t$. 

However, in the case where $\kappa \equiv 0$, we define $\tau \equiv 0$ and choose $\mathbf{N}, \mathbf{B}$ to be auto-parallel vector fields along $\bs{\gamma}$ such that the vectors $\mathbf{T}(t), \mathbf{N}(t), \text{ and } \mathbf{B}(t)$  form an orthonormal basis of $T_{\bs{\gamma}(t)}\sx^3$. Under this choice, the Frenet-Serret Formulas given in \eqref{eq-frenet} are again satisfied.

Note that we are not assuming that the manifold $M$ is orientable. In the case where $M$ is in fact oriented  (i.e., $M$ is orientable, and one of the two orientations is specified), there is a variant of the Frenet-Serret equations in which one assumes that the frame $\{\mathbf{T,N,B}\}$ is positively oriented. Then, one must allow the torsion $\tau$ to assume negative values. The Equations ~\ref{eq-frenet} continue to hold with this new interpretation. However, in this paper, we use the non-oriented form of the Frenet-Serret equations, where $\kappa$ and $\tau$ are always non-negative. Geometrically, this means that while considering helices in $\sx^3$, we disregard the chirality.

\subsection{The Frenet equations in \texorpdfstring{$\sx^3$}{S3}} \label{sec-frenetS3}
We begin by specializing the Frenet-Serret equations given in~\eqref{eq-frenet} to the case of the embedded sphere $\sx^3$ in $\rl^4$. Let 
    \[ \iota: \sx^3 \hookrightarrow \rl^4\]
be the natural embedding. Given a curve $\bs{\gamma}:I\to \sx^3$, where $I\subset\rl$ is an open interval, we may think of $\bs{\gamma}$ as a curve in $\rl^4$, by identifying $\bs{\gamma}$ with $\iota\circ\bs{\gamma}$. Similarly, given a vector field $\mathbf{v}$ along the curve $\bs{\gamma}$ which assigns to each point $t\in I$ a vector $\mathbf{v}(t)\in T_{\bs{\gamma}(t)}\sx^3$, we can identify $\mathbf{v}$ with the vector field $\iota_* \mathbf{v}$ along $\iota\circ\bs{\gamma}$, which assigns to the point $t\in I$ the vector $\iota_* \mathbf{v}(t)\in T_{\iota\circ\bs{\gamma}(t)}\rl^4$. In order to simplify notation, we adopt the following conventions:

\begin{enumerate}[leftmargin=*] \itemsep5pt
    
    \item Consistently identifying $\sx^3$ with the embedded image, we will omit the map $\iota$ and its pushforward $\iota_*$ from the notation. Thus, we will think of the $\bs{\gamma}$ in $\sx^3$ as a curve in $\rl^4$ whose image lies in $\sx^3$. Similarly, we will think of a vector field $\mathbf{v}$ in $\sx^3$ along $\bs{\gamma}$ as a vector field in $\rl^4$ along $\bs{\gamma}$ such that for each $t$, the vector $\mathbf{v}(t)\in T_{\bs{\gamma}(t)}\rl^4$ lies in the subspace $T_{\bs{\gamma}(t)}\sx^3$.
    
    \item We identify the tangent bundle $T\rl^4$ with $\rl^4\times\rl^4$. Therefore, all vector fields in $\rl^4$ may be identified with $\rl^4$-valued functions.
    
    \item Given a vector field $\mathbf{v}$ along a curve $\bs{\gamma}$ in $\sx^3$, by the previous two parts, we can identify it with a $\rl^4$-valued function. We will let $\mathbf{v}'$ denote its  derivative in the Euclidean space $\rl^4$, i.e., if $\mathbf{v}$ is represented using the natural coordinates as
        \[ \mathbf{v}(t)=(v_1(t), v_2(t), v_3(t), v_4(t)),\]
    where $v_j:I\to\rl^4$ is smooth, $j=1,\dots, 4$, then 
        \[ \mathbf{v}'(t)=(v_1'(t), v_2'(t), v_3'(t), v_4'(t)).\]
    Of course, this is nothing but a coordinate expression for the covariant derivative of the vector field $\mathbf{v}$ along $\bs{\gamma}$ with respect to the flat Euclidean metric of $\rl^4$.
    
\end{enumerate}

\begin{proposition}
Let $\bs{\gamma}:I\to \sx^3$ be a smooth curve in the three-sphere parametrized by arc length, and let $\mathbf{T,N,B}$ be its Frenet frame. Using the notational convention explained above, we think of $\mathbf{T,N,B}$ as functions from $I$ to $\rl^4$. Then, these vector valued functions satisfy the following differential equations:

\begin{align}
     \mathbf{T}'(t) + \bs{\gamma}(t) &= - \kappa(t) \mathbf{N}(t)\nonumber \\
     \mathbf{N}'(t) & = - \kappa(t) \mathbf{T}(t) + \tau(t) \mathbf{B}(t)\label{eq-frenets3} \\
     \mathbf{B}'(t) & = -\tau(t) \mathbf{N}(t)\nonumber
\end{align}
\end{proposition}

\begin{proof}
We begin by recalling the following fact from differential geometry \cite[Vol. III, p. 2]{spivak}. Let $M$ be an embedded submanifold of $\rl^N$, and for each point $x\in M$, let 
    \[ \mathcal{P}_x: T_x\rl^N\to T_xM\]
denote the orthogonal projection (where we identify $T_xM$ in the natural way with a subspace of $T_x\rl^N=\rl^N$).  We endow $M$ with the Riemannian metric induced by the Euclidean metric of $\rl^N$. Let $\bs{\gamma}:I\to M$ be a smooth curve in $M$, where $I\subset\rl$ is an open interval, and assume that $\bs{\gamma}$ is parametrized by arc length. If $\mathbf{v}$ is a vector field along $\bs{\gamma}$, it is well-known that the covariant derivative of $\mathbf{v}$ is given by
    \[ \frac{D}{dt} \mathbf{v}(t)= \mathcal{P}_{\bs{\gamma}(t)}\left(\mathbf{v}'(t)\right)\in T_{\bs{\gamma}(t)}M.\]
When $M$ is a hypersurface in $\rl^N$ and $x\in M$, we may write for $\mathbf{r}\in T_x\rl^N\simeq \rl^N$,
    \[ \mathcal{P}_x(\mathbf{r})= \mathbf{r}-\left(\mathbf{n}(x)\cdot \mathbf{r}\right) \mathbf{n}(x),\]
where $\mathbf{n}(x)$ denotes a unit vector normal to the hypersurface $M$ at the point $x$.
Consequently we obtain the following formula for differentiating a vector field $\mathbf{v}$ along the curve $\bs{\gamma}$:
    \[ \frac{D}{dt}\mathbf{v}(t)= \mathbf{v}'(t)-\left(\mathbf{n}(\bs{\gamma}(t))\cdot \mathbf{v}'(t)\right) \mathbf{n}(\bs{\gamma}(t)).\]
When $M=\sx^3$ in $\rl^4$, we may take for $\mathbf{x}\in \sx^3$
    \[ \mathbf{n}(\mathbf{x})= \mathbf{x},\]
so that

\begin{align}\label{eq-covderiv}
    \frac{D}{dt}\mathbf{v}(t)= \mathbf{v}'(t)-\left(\bs{\gamma}(t))\cdot \mathbf{v}'(t)\right) \bs{\gamma}(t)).
\end{align}

We now compute $\bs{\gamma}(t)\cdot \mathbf{v}'(t)$ when $\mathbf{v}$ is one of the Frenet frame vector fields $\mathbf{T,B,N}$. Note that the four vectors $\bs{\gamma}(t)$, $\mathbf{T}(t)$, $\mathbf{N}(t)$, and $\mathbf{B}(t)$ form an orthonormal set in $\mathbb{R}^4$.
Observe that for each $t$ we have the following equations:

\begin{align*}
    \mathbf{T}'(t) \cdot \bs{\gamma}(t) & = (\mathbf{T} \cdot \bs{\gamma})'(t)  - \mathbf{T}(t) \cdot \bs{\gamma}'(t) = 0 - (\mathbf{T}(t) \cdot \mathbf{T}(t)) = -1, \\
    \mathbf{N}'(t) \cdot \bs{\gamma}(t) & = (\mathbf{N} \cdot \bs{\gamma})'(t)  - \mathbf{N}(t) \cdot \bs{\gamma}'(t) = 0 - (\mathbf{N}(t) \cdot \mathbf{T}(t)) = 0, \\
    \mathbf{B}'(t) \cdot \bs{\gamma}(t) & = (\mathbf{B} \cdot \bs{\gamma})'(t)  - \mathbf{B}(t) \cdot \bs{\gamma}'(t) = 0 - (\mathbf{B}(t) \cdot \mathbf{T}(t)) = 0.
\end{align*}

\noindent In the first equation, we have used the fact that $\mathbf{T}=\bs{\gamma}'$. Using Equation~\eqref{eq-covderiv} and the above compuations, we  obtain the following represenations of the covariant derivatives of the Frenet frame:

\begin{align*}
    \frac{D}{dt}\mathbf{T}(t) & = \mathbf{T}'(t) + \bs{\gamma}(t) \\
    \frac{D}{dt}\mathbf{N}(t) & = \mathbf{N}'(t) \\
    \frac{D}{dt}\mathbf{B}(t) & = \mathbf{B}'(t).
\end{align*}

Combining these with the Frenet-Serret equations \eqref{eq-frenet} in a Riemannian three-manifold, Equations~\eqref{eq-frenets3} follow.
\end{proof}

\section{Lissajous curves in \texorpdfstring{$\sx^3$}{S3}} \label{sec-lissajouscurvesonS3}
In this section, we prove a few results that will be needed to complete the proof of Theorem~\ref{thm-main}. The following lemma will be required. 
\begin{lemma} \label{fourierTrick}
Let $\alpha_0, \alpha_1, \cdots, \alpha_N$ be distinct non-negative real numbers, and suppose that for each $t\geq 0$, we have
\begin{align} \label{f(t)}
    \sum_{k = 0}^N \left(a_k \sin{(\alpha_k t)} + b_k \cos{(\alpha_k t)} \right)=0,
\end{align}
where the coefficients $a_k, b_k$ are complex numbers. Then we have
$a_k=b_k=0$ for each $k$.
\end{lemma}

\begin{proof}
We can assume without loss of generality that $\alpha_0=0$ (simply take $a_0=b_0=0$). For $k=1,\dots, N$, let us set $\alpha_{-k}= -\alpha_k$. Then, Equation \eqref{f(t)} takes the form
\begin{align}
    \sum_{k=-N}^N c_k e^{i\alpha_k t}=0 \label{eq-complex}
\end{align}
where
\begin{align*}
    c_k = \begin{cases} 
                \cfrac{a_{\abs{k}} + i b_{\abs{k}}}{2i}, &k>0\\
                \quad \quad a_0, &k=0\\
                \cfrac{-a_{\abs{k}} + i b_{\abs{k}}}{2i}, &k<0.
          \end{cases}
\end{align*}
For each $k \geq 0$, it follows that $c_k=c_{-k}=0$ if and only if $a_k=b_k=0$.

Fix an integer  $\ell$, $\abs{\ell}\leq N$,  and multiply both sides of Equation \eqref{eq-complex} by $e^{-i\alpha_\ell t}$. Integrating on the interval $[0,T]$ and dividing by $T$, we see that for each $T\geq 0$ we have
\begin{align}\label{average}
    \sum_{\substack{k=-N\\ k\not=\ell}}^N \frac{c_k}{T} \int_0^T e^{i(\alpha_k-\alpha_\ell)t} dt +c_\ell =0.
\end{align}
Note, however, that if $k\not = \ell$, we have
    \[ \abs{\int_0^T e^{i(\alpha_k-\alpha_\ell)t}dt} \leq \abs{\frac{e^{i(\alpha_k-\alpha_\ell)T}-1}{i(\alpha_k-\alpha_\ell)}}\leq \frac{2}{\abs{\alpha_k-\alpha_\ell}}.\]
Since for each $k$, this is bounded independently of $T$, as $T\to \infty$, each term in the first sum of Equation \eqref{average} goes to 0, which shows that $c_\ell=0$. Therefore, $a_\ell = b_\ell = 0$. Since $\ell$ is arbitrary, the lemma is proved.
\end{proof}
We will also need the following proposition.
\begin{proposition} \label{prop-lissajous}
    Suppose that the Lissajous curve given by
    Equation \eqref{eq-lissajous} lies in $\sx^3$.
    \begin{itemize}[leftmargin=7mm] \itemsep5pt
        \item[(a)] If $\omega_1 \neq 0$ and $\bs{\gamma}$ has constant speed, then the coefficient vectors of $\bs{\gamma}$ satisfy the following relations:\vspace{1mm}
        \begin{enumerate}[leftmargin=*] \itemsep5pt
            \item $\ba_1, \bb_1, \ba_2, \bb_2$ are orthogonal
            \item $\abs{\ba_1} = \abs{\bb_1}$ and $\abs{\ba_2} = \abs{\bb_2}$
            \item $|\ba_1|^2 + |\ba_2|^2 = 1$.
        \end{enumerate}
    \item[(b)] If $\omega_1 = 0$, then:\vspace{1mm}
        \begin{enumerate}[leftmargin=*] \itemsep5pt
            \item $\ba_1, \ba_2, \bb_2$ are orthogonal
            \item $\abs{\ba_2} = \abs{\bb_2}$
            \item $|\ba_1|^2 + |\ba_2|^2 = 1$.
        \end{enumerate}
    \end{itemize}

\end{proposition}
 
\begin{proof}
For use in the later portions of this proof, we will first compute $\abs{\bs{\gamma}(t)}^2$.  Using Equation \eqref{eq-lissajous} and the fact that $\bs{\gamma}$ lies in $\sx^3$, for each $t$ we have
\begin{align}
    \bs{\gamma}(t) \cdot \bs{\gamma}(t) = 1 = & |\mathbf{A}_1|^2\cos^2(\omega_1 t) + |\mathbf{B}_1|^2 \sin^2(\omega_1 t) + |\mathbf{A}_2|^2 \cos^2(\omega_2 t) + |\mathbf{B}_2|^2 \sin^2(\omega_2 t) \nonumber\\
    &+2(\mathbf{A}_1 \cdot \mathbf{B}_1) \cos(\omega_1 t) \sin(\omega_1 t) + 2(\mathbf{A}_1 \cdot \mathbf{A}_2) \cos(\omega_1 t)\cos(\omega_2 t)\nonumber \\
    &+2(\mathbf{A}_1 \cdot \mathbf{B}_2) \cos(\omega_1 t) \sin(\omega_2 t) + 2(\mathbf{B}_1 \cdot \mathbf{A}_2) \sin(\omega_1 t) \cos(\omega_2 t)\nonumber\\
    &+2(\mathbf{B}_1 \cdot \mathbf{B}_2) \sin(\omega_1 t) \sin(\omega_2 t) + 2(\mathbf{A}_2 \cdot \mathbf{B}_2) \cos(\omega_2 t) \sin(\omega_2 t)\nonumber \\
    = & \left( \frac{|\mathbf{A}_1|^2 + |\mathbf{B}_1|^2 + |\mathbf{A}_2|^2 + |\mathbf{B}_2|^2}{2} \right) + \left( \frac{|\mathbf{A}_1|^2 - |\mathbf{B}_1|^2}{2} \right) \cos{(2 \omega_1 t)}\nonumber \\
    & + \left( \mathbf{A}_1 \cdot \mathbf{B}_1 \right) \sin{(2 \omega_1 t)} + \left( \frac{|\mathbf{A}_2|^2 - |\mathbf{B}_2|^2}{2} \right) \cos{(2 \omega_2 t)} + \left( \mathbf{A}_2 \cdot \mathbf{B}_2 \right) \sin{(2 \omega_2 t)}\nonumber \\
    & + \left( \mathbf{A}_1 \cdot \mathbf{A}_2 - \mathbf{B}_1 \cdot \mathbf{B}_2 \right) \cos{((\omega_1 + \omega_2) t)} + \left( \mathbf{A}_1 \cdot \mathbf{B}_2 + \mathbf{B}_1 \cdot \mathbf{A}_2 \right) \sin{((\omega_1 + \omega_2) t)}\nonumber \\
    & + \left( \mathbf{B}_1 \cdot \mathbf{B}_2 + \mathbf{A}_1 \cdot \mathbf{A}_2 \right) \cos{((\omega_2 - \omega_1) t)} + \left( \mathbf{A}_1 \cdot \mathbf{B}_2 - \mathbf{B}_1 \cdot \mathbf{A}_2 \right) \sin{((\omega_2 - \omega_1) t)}\label{eq-gdotg}
\end{align}
First, we prove part (a) of the proposition. Let u begin by assuming that $\omega_2 \neq 3 \omega_1$. Since $\omega_2 \neq 3 \omega_1$ and $\omega_1 \neq 0$, we see that the five numbers 
$$0,\ 2 \omega_1,\ 2 \omega_2,\ \omega_2 - \omega_1,\ \text{ and }\ \omega_2 + \omega_1$$ 
are all distinct. By Lemma \ref{fourierTrick}, each of the coefficients in the expression for $\bs{\gamma}(t) \cdot \bs{\gamma}(t) -1$ vanishes, as in the left hand side of Equation \eqref{f(t)}. Thus, from the coefficients of Equation~\eqref{eq-gdotg}, we obtain: 
\begin{align*}
    \left.\begin{aligned}
        &\frac{|\mathbf{A}_1|^2 + |\mathbf{B}_1|^2 + |\mathbf{A}_2|^2 + |\mathbf{B}_2|^2}{2} = 1 \\
        &\cfrac{|\mathbf{A}_1|^2 - |\mathbf{B}_1|^2}{2} = 0\\
        &\mathbf{A}_1 \cdot \mathbf{B}_1 = 0\\
        &\cfrac{|\mathbf{A}_2|^2 - |\mathbf{B}_2|^2}{2} = 0\\
        &\mathbf{A}_2 \cdot \mathbf{B}_2 = 0\\
        &\mathbf{A}_1 \cdot \mathbf{A}_2 - \mathbf{B}_1 \cdot \mathbf{B}_2 = 0\\
        &\mathbf{A}_1 \cdot \mathbf{B}_2 + \mathbf{B}_1 \cdot \mathbf{A}_2 = 0\\
        &\mathbf{B}_1 \cdot \mathbf{B}_2 + \mathbf{A}_1 \cdot \mathbf{A}_2 = 0\\
        &\mathbf{A}_1 \cdot \mathbf{B}_2 - \mathbf{B}_1 \cdot \mathbf{A}_2 = 0
    \end{aligned}\right\}
\end{align*}
which yields the following:
\begin{align}
    &|\mathbf{A}_1|^2 + |\mathbf{B}_1|^2 + |\mathbf{A}_2|^2 + |\mathbf{B}_2|^2 = 2 \label{DC}\\
    &|\mathbf{A}_1| = |\mathbf{B}_1| \label{a1=a2}\\
    &|\mathbf{A}_2| = |\mathbf{B}_2| \label{a3=a4}\\
    &\mathbf{A}_1 \cdot \mathbf{B}_1 = \mathbf{A}_1 \cdot \mathbf{A}_2 = \mathbf{A}_1 \cdot \mathbf{B}_2 = \mathbf{B}_1 \cdot \mathbf{A}_2 = \mathbf{B}_1 \cdot \mathbf{B}_2  = \mathbf{A}_2 \cdot \mathbf{B}_2 = 0 \label{orth}
\end{align}
Equation \eqref{orth} shows that the vectors $\ba_1, \bb_1, \ba_2,\ \text{and }\bb_2$ are orthogonal, which is conclusion 1 of the proposition. Moreover, Equations \eqref{a1=a2} and \eqref{a3=a4} are precisely conclusion 2 of the proposition. Further, recognize that by using Equations (\ref{DC}-\ref{a3=a4}), we get 
\begin{align}
    |\mathbf{A}_1|^2 + |\mathbf{A}_2|^2 = 1 \label{A1^2+A2^2=1}
\end{align}
which is conclusion 3 of the proposition.

To complete the proof of part (a) of the proposition, we now consider the case when $\omega_2 = 3 \omega_1$. Let us set $\omega_1 = \omega$ and thus, $\omega_2 = 3\omega$. Therefore, we have 
$$2 \omega_1 = 2\omega, \quad 2\omega_2 = 6 \omega, \quad \omega_2 - \omega_1 = 2 \omega, \text{ and } \omega_2 + \omega_1 = 4\omega.$$ 
Notice, $2\omega_1 = \omega_2 - \omega_1 = 2\omega$, so in Equation~\eqref{eq-gdotg} there are only 4. Therefore, the relation
$\bs{\gamma}(t) \cdot \bs{\gamma}(t) = 1$  applied to Equation \eqref{eq-gdotg} and Lemma~\ref{fourierTrick} now give, 
\begin{align}
    \label{gr1}
    &|\textbf{A}_1|^2 + |\textbf{A}_2|^2 + |\textbf{B}_1|^2 + |\textbf{B}_2|^2  = 2\\
    \label{gr2}
    &|\textbf{A}_1|^2 - |\textbf{B}_1|^2 + 2(\textbf{A}_1 \cdot \textbf{A}_2 + \textbf{B}_1 \cdot \textbf{B}_2) = 0\\
    \label{gr3}
    &|\textbf{A}_2| = |\textbf{B}_2|\\
    \label{gr4}
    &\textbf{A}_1 \cdot \textbf{B}_1+ \textbf{A}_1 \cdot \textbf{B}_2 - \textbf{B}_1 \cdot \textbf{A}_2 = 0\\
    \label{gr5}
    &\textbf{A}_2 \cdot \textbf{B}_2 = 0\\
    \label{gr6}
    &\textbf{A}_1 \cdot \textbf{A}_2 = \textbf{B}_1 \cdot \textbf{B}_2\\
    \label{gr7}
    &\textbf{A}_1 \cdot \textbf{B}_2 = -\textbf{B}_1 \cdot \textbf{A}_2
\end{align}
Since $\bs{\gamma}$ has constant speed, there exists a $C>0$ such that for all $t$, we have $\abs{\bs{\gamma}'(t)}=C$. The relation $\bs{\gamma}'(t) \cdot \bs{\gamma}'(t) = C^2$ yields (using Equation \eqref{eq-gdotg}) 
\begin{align}
0 = &\left(\frac{\omega^2|\mathbf{B}_1|^2 + 9\omega^2|\mathbf{B}_2|^2 + \omega^2 |\mathbf{A}_1|^2 + 9\omega^2|\mathbf{A}_2|^2}{2} - C^2 \right)\nonumber \\
&+ \left(\frac{\omega^2|\mathbf{B}_1|^2 - \omega^2|\mathbf{A}_1|^2}{2} + 3 \omega^2 \mathbf{B}_1 \cdot \mathbf{B}_2 + 3 \omega^2 \mathbf{A}_1 \cdot \mathbf{A}_2 \right)\cos(2 \omega t)\nonumber\\
& + 9 \omega^2\left(\frac{|\mathbf{B}_2|^2 - |\mathbf{A}_2|^2}{2} \right) \cos(6 \omega t)
- (\omega^2(\mathbf{A}_1 \cdot \mathbf{B}_1) +3 \omega^2 (\mathbf{B}_1 \cdot \mathbf{A}_2) - 3\omega^2(\mathbf{B}_2 \cdot \mathbf{A}_1) )\sin(2 \omega t)\nonumber\\  &-9 \omega^2 (\mathbf{A}_2 \cdot \mathbf{B}_2) \sin(6 \omega t)
+ 3 \omega^2(\mathbf{B}_1 \cdot \mathbf{B}_2 - \mathbf{A}_1 \cdot \mathbf{A}_2 ) \cos(4 \omega t)  -3 \omega^2(\mathbf{B}_1 \cdot \mathbf{A}_2 + \mathbf{A}_1 \cdot \mathbf{B}_2) \sin(4 \omega t) \label{eq-gp}
\end{align}
Using Lemma~\ref{fourierTrick}, this gives the relations 
\begin{align}
    \label{gpr1}
    &\omega^2(|\mathbf{B}_1|^2 + 9|\mathbf{B}_2|^2 + |\mathbf{A}_1|^2 + 9|\mathbf{A}_2|^2) = 2C^2\\
    \label{gpr2}
    &|\mathbf{B}_1|^2 -|\mathbf{A}_1|^2 + 6(\mathbf{B}_1 \cdot \mathbf{B}_2 + \mathbf{A}_1 \cdot \mathbf{A}_2) = 0\\
    \label{gpr3}
    &|\mathbf{B}_2| =|\mathbf{A}_2|\\
    \label{gpr4}
    &\mathbf{A}_1 \cdot \mathbf{B}_1 +3(\mathbf{B}_1 \cdot \mathbf{A}_2 - \mathbf{B}_2 \cdot \mathbf{A}_1 ) = 0\\
    \label{gpr5}
    &\mathbf{A}_2 \cdot \mathbf{B}_2  = 0\\
    \label{gpr6}
    &\mathbf{B}_1 \cdot \mathbf{B}_2 = \mathbf{A}_1 \cdot \mathbf{A}_2\\
    \label{gpr7}
    &\mathbf{B}_1 \cdot \mathbf{A}_2 =- \mathbf{A}_1 \cdot \mathbf{B}_2
\end{align}
Comparing Equations (\ref{gr1} - \ref{gr7}) and Equations (\ref{gpr1} - \ref{gpr7}), we see that we have obtained three new relations, which are Equations~\eqref{gpr1},~\eqref{gpr2}~and~\eqref{gpr4}. Combining Equations \eqref{gr2} with \eqref{gpr2} we see that 
    \[\abs{\ba_1} = \abs{\bb_1} \text{ and } \ba_1 \cdot \ba_2 = -\bb_1 \cdot \bb_2.\]
Similarly, Equations \eqref{gr4} and \eqref{gpr4} imply that
    \[\ba_1 \cdot \bb_1 = 0 \text{ and } \ba_1 \cdot \bb_2 = \bb_1 \cdot \ba_2.\]
Combining these last two relations with Equations (\ref{gr1} - \ref{gr7}), we get that the coefficient vectors $\ba_1, \bb_1, \ba_2, \bb_2$ are mutually orthogonal, $\abs{\ba_1}= \abs{\bb_1}$, and $\abs{\ba_2}= \abs{\bb_2}$. Conclusions (1), (2) and (3) follow again.

Now we will prove part (b) of the proposition.  We set $\omega_1 = 0$ and $\omega_2 = \omega$ in Equation~\eqref{eq-gdotg}.  Then in Equation~\eqref{eq-gdotg} we have the following three distinct frequencies:
    \[0 = 2 \omega_1, \quad 2 \omega_2 = 2 \omega, \quad \omega_2 + \omega_1 = \omega_2 - \omega_1 = \omega\]
Therefore, by Lemma~\ref{fourierTrick} we have
\begin{align}
    &2\abs{\ba_1}^2 + \abs{\ba_2}^2 + \abs{\bb_2}^2 = 2\label{eq-degen1}\\
    &\abs{\ba_2} = \abs{\bb_2}\label{eq-degen2}\\
    &\ba_1 \cdot \ba_2  = 0\nonumber,\ \ba_1 \cdot \bb_2 = 0\nonumber, \text{ and }\ba_2 \cdot \bb_2 = 0\nonumber
\end{align}
Equations~\eqref{eq-degen1}~and~\eqref{eq-degen2} are precisely conclusions 1 and 2 of part (b) of the proposition. Additionally, combining Equations~\eqref{eq-degen1},~and~\eqref{eq-degen2} we have conclusion 3 of part (b) of the proposition.

\end{proof}
 In connection with the proof of part (a) of the above proposition, we note that if $\omega_2 = 3 \omega_1$, one can construct Lissajous curves in $\sx^3$ (of non-constant speed) for which the coefficient vectors are not orthogonal.

\section{Proof of Theorem~\ref{thm-main}}\label{sec-proofthm1}
\subsection{Part 1} \label{sec-proofthm1part1}
Adjoining the relation $\bs{\gamma}'(t) = \mathbf{T}(t)$ to the Frenet-Serret equations (Equation \eqref{eq-frenets3}) we obtain the system of equations
\begin{alignat}{6}
     \bs{\gamma}'(t) &= && &&\phantom{-}\phantom{\kappa}\mathbf{T}(t) && && \nonumber \\
     \mathbf{T}'(t)  &= &&-\bs{\gamma}(t) && &&-\kappa \mathbf{N}(t) && \nonumber \\
     \mathbf{N}'(t)  &= && &&-\kappa \mathbf{T}(t) && &&+\tau \mathbf{B}(t) \label{eq-augfs} \\
     \mathbf{B}'(t)  &= && && &&-\tau \mathbf{N}(t) && \nonumber
\end{alignat}
where $\kappa$ and $\tau$ are the given constants.  We now rewrite these equations in matrix form.  Note that the four vectors $\bs{\gamma}(t), \mathbf{T}(t),\mathbf{N}(t), \mathbf{B}(t)$ form an orthonormal basis of $\rl^4$. Let $\mathbf{X}(t)$ denote the $4\times 4$ matrix whose rows are  these four vectors.  Then for each $t$, the matrix $\mathbf{X}(t)$ is orthogonal. When the curvature $\kappa$ and the torsion $\tau$ are constants, the augmented Frenet-Serret equations given by the set of Equations in \eqref{eq-augfs} in the sphere Equation may be written in matrix form as
\begin{align}\label{eq-frenet-constant}
    \mathbf{X}'(t)= \mathbf{C}\cdot\mathbf{X}(t)
\end{align}
where $\mathbf{C}$ denotes the skew-symmetric matrix
\begin{align}\label{eq-C}
     \mathbf{C} = \begin{bmatrix}
        0 & 1 & 0 & 0 \\
        -1 & 0 & \kappa & 0 \\
        0 & -\kappa & 0 & \tau \\
        0 & 0 & -\tau & 0
    \end{bmatrix}
\end{align}

From the theory of ordinary differential equations we know that the solution to the constant coefficient system presented in Equation \eqref{eq-frenet-constant} exists for all $t$ and is given by

\begin{align} \label{eq-ODEsol}
    \mathbf{X}(t) = e^{t\mathbf{C}}\cdot \mathbf{X}(0).
\end{align}
This proves part 1 of the theorem.

\subsection{Part 2} \label{sec-proofthm1part2}
In order to calculate the matrix exponential $e^{t\mathbf{C}}$, we first recognize that since $\mathbf{C}$ is skew-symetric, it can be diagonalized and thus written as
\begin{align*}
    \mathbf{C} =  \mathbf{P} \left(i \mathbf{D}\right) \mathbf{P}^{-1}
\end{align*}
where $\mathbf{D}$ is a diagonal matrix with real entries of the form
\begin{align} \label{eq-D}
    \mathbf{D} ={\rm  diag} ( \omega_1, - \omega_1, \omega_2, -\omega_2 ),
\end{align}
where $\omega_1,\omega_2\geq 0$.  
This is because the eigenvalues of the real skew-symmetric matrix $\mathbf{C}$ are purely imaginary and occur in complex conjugate pairs.
Therefore,
\begin{align} \label{similar}
    e^{t\mathbf{C}} = \mathbf{P} e^{i t \mathbf{D}} \mathbf{P}^{-1}.
\end{align}
From Equation~\eqref{eq-D} it follows that
$$e^{it\mathbf{D}} = {\rm diag} (e^{i \omega_1 t}, e^{-i \omega_1 t}, e^{i \omega_2 t}, e^{-i \omega_2 t}).$$
Since $\bs{\gamma}(t)$ is the first row of the matrix $\mathbf{X}(t) = \mathbf{P} e^{i t \mathbf{D}} \mathbf{P}^{-1} \cdot \mathbf{X}(0)$ it follows that
$$\bs{\gamma}(t) = \cos{(\omega_1 t)} \mathbf{A}_1 + \sin{(\omega_1 t)} \mathbf{B}_1 + \cos{(\omega_2 t)} \mathbf{A}_2 + \sin{(\omega_2 t)} \mathbf{B}_2$$
where the coefficient vectors $\ba_1, \bb_1, \ba_2, \text{ and } \bb_2$ are constant vectors in $\rl^4$.  This proves part 2 of the theorem.

\subsection{Part 3} \label{sec-proofthm1part3}
The diagonal entries of $i\mathbf{D}$, where $\mathbf{D}$ is as in Equation~\eqref{eq-D}, are the eigenvalues of the matrix $\mathbf{C}$ of Equation~\eqref{eq-C}.  We find them by solving the characteristic equation
\begin{align}
    \det\left(\mathbf{C} - x\mathbf{I}\right) = x^4 + (\kappa^2 + \tau^2 + 1) x^2 + \tau^2 = x^4 + \chi^2 x^2 + \tau^2 = 0 \label{eq-discriminant}
\end{align}
with $\chi$ as in Equation~\eqref{eq-chidef}.
The solutions of the  characteristic equation are 
$$x = \pm i\omega_1 \text{ or } x = \pm i \omega_2,$$
where $\omega_1,\omega_2$ are as in Equations~\eqref{eq-omega1}~and~\eqref{eq-omega2}.  This proves part 3 of the theorem.

\subsection{Part 4} \label{sec-proofthm1part4}
Since $\kappa > 0$, we have
\begin{align}
    \chi^4 - 4\tau^2 &= (\kappa^2 + \tau^2 + 1)^2 - 4\tau^2\nonumber \\
                     &= \kappa^4 + (\tau^2 -1)^2 + 2 \kappa^2 \tau^2 + 2\kappa^2\nonumber \\
                     &> (\tau^2 -1)^2 \label{eq-discriminant2}.
\end{align}
Therefore, from Equations~\eqref{eq-omega1}~and~\eqref{eq-omega2} we see that $\omega_2 > \omega_1$. by definition of $\chi$ in \eqref{eq-chidef} we have, 
\begin{align}\label{eq-chi}
    \chi^2 = \kappa^2 + \tau^2 + 1 > 1 + \tau^2.
\end{align}
Combining Equations~\eqref{eq-discriminant2}~and~\eqref{eq-chi} we have 
\begin{align}
    \chi^2 + \sqrt{\chi^4 - 4\tau^2} &> (1 + \tau^2) + \sqrt{(\tau^2 - 1)^2} \label{eq-chi2}\\
    &= 1 + \tau^2 + |\tau^2 - 1| \nonumber\\
    &= \begin{cases}\label{eq-eq1234}
        2 \tau^2, & \tau \geq 1\\
        2, &\tau < 1
    \end{cases}
\end{align}
Therefore, $\chi^2 + \sqrt{\chi^4 - 4 \tau^2} > 2$, and we have
\begin{align} \label{w2geq1}
    \omega_2^2 = \frac{\chi^2 + \sqrt{\chi^4 - 4 \tau^2}}{2} > 1.
\end{align}
Then, by making use of Equation \eqref{eq-eq1234},
\begin{align}
    \omega_1^2 =\frac{\chi^2 - \sqrt{\chi^4 - 4 \tau^2}}{2} = \frac{2 \tau^2}{\chi^2 + \sqrt{\chi^2 - 4 \tau^2}} < \begin{cases}
        1, &\tau \geq 1\\
        \tau^2, &\tau < 1
    \end{cases}
\end{align}
Thus, 
\begin{align}
    \omega_1 < 1 \label{w1leq1}
\end{align}
This proves part 4 of the theorem.

\subsection{Part 5} \label{sec-proofthm1part5}
Note that $\abs{\bs{\gamma}(t)} = 1$ for all $t$ and $\abs{\bs{\gamma}'(t)} = 1$ for all $t$ as well, since $\bs{\gamma}$ lies in $\sx^3$ and has unit speed. Since $\tau \neq 0$ by Equation~\eqref{eq-omega1}, we know that $\omega_1 > 0$.  Therefore, by part (a) of Proposition~\ref{prop-lissajous},
\begin{enumerate}[leftmargin=*] \itemsep5pt
    \item The coefficient vectors $\ba_1, \bb_1, \ba_2, \bb_2$ are mutually orthogonal, which is one of the conclusions of part 5 of Theorem \ref{thm-main},
    
    \item $\abs{\ba_1}= \abs{\bb_1} \text{ and } \abs{\ba_2}= \abs{\bb_2}$ which is part of the content of Equations~\eqref{eq-a1b1}~and~\eqref{eq-a2b2}. We will however need to work further to obtain the remaining content of these equations and to therefore complete the proof of part 5,
    
    \item We have that
        \begin{align} \label{eq-A1A21}
            \abs{\ba_1}^2 + \abs{\ba_2}^2 = 1.
        \end{align} 
\end{enumerate}
Now, let $\bs{\alpha}(t) = \bs{\gamma}'(t)$. Then, $\abs{\bs{\alpha}(t)} = 1$ (since $\bs{\gamma}(t)$ is parameterized by arc length) and differentiating Equation~\eqref{eq-lissajous}, we see that $\bs{\alpha}(t)$ may be represented as
\begin{align*}
    \bs{\alpha}(t) &= \omega_1  \cos(\omega_1 t)\mathbf{B}_1  - \omega_1 \sin(\omega_1 t)\mathbf{A}_1  + \omega_2 \cos(\omega_2 t)\mathbf{B}_2  - \omega_2 \sin(\omega_2 t)\mathbf{A}_2\\
    & = \cos(\omega_1 t)\mathbf{P}_1 + \sin(\omega_1 t) \mathbf{Q}_1 + \cos(\omega_2 t) \mathbf{P}_2 + \sin(\omega_2 t) \mathbf{Q}_2,
\end{align*}
where $\mathbf{P}_1 = \omega_1 \mathbf{B}_1$, $\mathbf{Q}_1 =- \omega_1 \mathbf{A}_1$, $\mathbf{P}_2 = \omega_2 \mathbf{B}_2$, and $\mathbf{Q}_2 = -\omega_2 \mathbf{A}_2$. This shows that $\bs{\alpha}$ is a Lissajous curve in $\sx^3$.

Now, we claim that $\abs{\bs{\alpha}'(t)}$ is constant independently of $t$. Recall, that $$\bs{\alpha}'(t) = \bs{\gamma}''(t) = \mathbf{T}'(t)$$
where $\mathbf{T}$ is the tangent vector field in the Frenet Frame $(\mathbf{T,N,B})$. Therefore, by the first equation in \eqref{eq-frenets3}, we have 
$$\bs{\gamma}''(t) = - \kappa \mathbf{N}(t) - \bs{\gamma}(t),$$
which yields, 
\begin{align}
    \bs{\alpha}'(t) \cdot \bs{\alpha}'(t) &= \kappa^2 + 1 + 2\kappa(\mathbf{N}(t)\cdot \bs{\gamma}(t))\nonumber \\
                                        & = \kappa^2 + 1
\end{align}
where the term $\mathbf{N}(t)\cdot \bs{\gamma}(t) = 0$ since $\mathbf{N}(t) \in T_{\bs{\gamma}(t)}\sx^3$. We may now apply conclusion 3 of part (a) of Proposition \ref{prop-lissajous} to obtain that $\abs{\mathbf{P}_1}^2 + \abs{\mathbf{P}_2}^2 = 1$, which is equivalent to the statement that 
\begin{align} \label{eq-w1a1}
    \omega_1^2 \abs{\ba_1}^2 + \omega_2^2 \abs{\ba_2}^2 = 1.
\end{align}
Combining Equations \eqref{eq-A1A21} and \eqref{eq-w1a1}, we get
$$\omega_1^2\abs{\ba_1}^2 + \omega_2^2(1 - \abs{\ba_1}^2) = 1$$
and
$$\omega_1^2(1- \abs{\ba_2}^2) + \omega_2^2 \abs{\ba_2}^2 = 1.$$
Solving these equations for $\abs{\ba_1}^2$ and $\abs{\ba_2}^2$, we obtain Equations \eqref{eq-a1b1} and \eqref{eq-a2b2}.

\subsection{Part 6}
If $\tau = 0$, by Equation~\eqref{eq-omega1}, we have $\omega_1 = 0$. We set $\omega_2 = \omega = \sqrt{\kappa^2 + 1}$ by Equation~\eqref{eq-omega2} and then by parts 1,2, and 3 of this theorem, proved above, $\bs{\gamma}$ is given by
    \[\bs{\gamma}(t) = \ba_1 + \cos(\omega t) \ba_2 + \sin(\omega t) \bb_2.\]
Furthermore, by part (b) of Proposition~\ref{prop-lissajous} we know that $\abs{\ba_2} = \abs{\bb_2}$, that the coefficient vectors $\ba_1, \ba_2, \text{ and } \bb_2$ are mutually orthogonal. Note that, 
    \[\bs{\gamma}'(t) = -\omega\sin(\omega t) \ba_2 + \omega \cos(\omega t)
    \bb_2.\]
Since $\bs{\gamma}'(t)$ has unit speed, we have 
    \[\bs{\gamma}'(t) \cdot \bs{\gamma}'(t) = 1 = \omega^2 \abs{\ba_2}^2\sin^2(\omega t) + \omega^2 \abs{\bb_2}^2 \cos^2(\omega t) = \omega^2 \abs{\ba_2}^2,\]
which implies, 
\begin{align}\label{eq-degenmag2}
    \abs{\ba_2} = \frac{1}{\omega}.
\end{align}
Combining Equation~\eqref{eq-degenmag2} with conclusion 3 of part (b) of Proposition~\ref{prop-lissajous}, we get that
\begin{align}\label{eq-degenmag1}
    \abs{\ba_1} = \sqrt{1 -\frac{1}{\omega^2}}.
\end{align}

\section{Proof of Theorem ~\ref{thm-existence}}\label{sec-proofthm2}
\subsection{Part 1} \label{sec-proofthm2part1}
Suppose that we have two helices $\bs{\alpha}$ and $\bs{\beta}$ in $\sx^3$ with the same curvature $\kappa \geq 0$ and torsion $\tau \geq 0$.  Then, by part 3 of Theorem~\ref{thm-main}, the fundamental angular frequencies $\omega_1 \text{ and } \omega_2$ of these two curves are the same. Thus, the curves are represented as
\begin{align*}
    \bs{\alpha}(t) &= \cos(\omega_1 t) \ba_1 + \sin(\omega_1 t) \bb_1 + \cos(\omega_2 t) \ba_2 + \sin(\omega_2 t) \bb_2\\
    \bs{\beta}(t)  &= \cos(\omega_1 t) \mathbf{C}_1 + \sin(\omega_1 t) \mathbf{D}_1 + \cos(\omega_2 t) \mathbf{C}_2 + \sin(\omega_2 t) \mathbf{D}_2
\end{align*}
If $\tau \neq 0$, by part 5 of Theorem~\ref{thm-main}, we also know that 
\begin{align*}
    \abs{\ba_1} &= \abs{\bb_1} = \abs{\mathbf{C}_1} = \abs{\mathbf{D}_1} = \sqrt{\cfrac{1 - \omega_2^2}{\omega_1^2 - \omega_2^2}},\\
    \abs{\ba_2} &= \abs{\bb_2} = \abs{\mathbf{C}_2} = \abs{\mathbf{D}_2} = \sqrt{\cfrac{1 - \omega_1^2}{\omega_2^2 - \omega_1^2}},
\end{align*}
and that the sets of vectors $\{\ba_1, \bb_1, \ba_2, \bb_2\}$ are mutually orthogonal and $\{\mathbf{C}_1, \mathbf{D}_1, \mathbf{C}_2, \mathbf{D}_2\}$ are also mutually orthogonal. Therefore, there exists an orthogonal map, $G: \rl^4 \to \rl^4$ in $O(4)$ such that $G(\ba_1) = \mathbf{C}_1$, $G(\bb_1) = \mathbf{D}_1$, $G(\ba_2) = \mathbf{C}_2$, and $G(\bb_2) = \mathbf{D}_2$. Then $f = G|_{\sx^3}$ is an isometry of $\sx^3$ and it is clear that $\bs{\beta} = f \circ \bs{\alpha}$.

If $\tau = 0$, then by part 6 of Theorem~\ref{thm-main}, we know that $\bs{\alpha}$ and $\bs{\beta}$ take the form
\begin{align*}
    \bs{\alpha}(t) &= \ba_1  + \cos(\omega t)\ba_2  + \sin(\omega t)\bb_2 \\
    \bs{\beta}(t)  &= \mathbf{C}_1 + \cos(\omega t)\mathbf{C}_2  + \sin(\omega t)\mathbf{D}_2 
\end{align*}
where $\omega_1 = 0$ and $\omega = \omega_2 = \sqrt{\kappa^2 + 1}$ by Equations~\eqref{eq-omega1}~and~\eqref{eq-omega2}. Furthermore, by part 6 of Theorem~\ref{thm-main}, we know that
\begin{align*}
    \abs{\ba_2} = \abs{\bb_2} = \abs{\mathbf{C}_2} = \abs{\mathbf{D}_2} = \frac{1}{\omega}, \quad \abs{\ba_1} = \abs{\mathbf{C}_1} = \sqrt{1 -\cfrac{1}{\omega^2}},
\end{align*}
and that the sets of vectors $\{\ba_1, \ba_2, \bb_2\}$ are mutually orthogonal and $\{\mathbf{C}_1, \mathbf{C}_2, \mathbf{D}_2\}$ are also mutually orthogonal. Therefore, there again exists an orthogonal map, $G : \rl^4 \to \rl^4$ in $O(4)$ such that $G(\ba_1) = \mathbf{C}_1$, $G(\ba_2) = \mathbf{C_2}$, $G(\bb_2) = \mathbf{D}_2$. Then $f = G|_{\sx^3}$ is again an isometry of $\sx^3$ and it is clear that $\bs{\beta} = f \circ \bs{\alpha}$.

\subsection{Part 2} \label{sec-proofthm2part2}
 Let $\bs{\gamma}$ be a helix in $\sx^3$ which can be written, thanks to Theorem \ref{thm-main}, in the form of Equation \eqref{eq-lissajous}:
    \[\bs{\gamma}(t) =\cos(\omega_1 t) \mathbf{A}_1  + \sin(\omega_1 t)\mathbf{B}_1  + \cos(\omega_2 t)\mathbf{A}_2 +\sin(\omega_2 t)\mathbf{B}_2\]
Now suppose that $\bs{\gamma}$ is periodic with period $T$. Then, for each $t \in \rl$, we have
    \[\bs{\gamma}(t) = \bs{\gamma}(t + T).\]
First, let us assume that $\tau \neq 0$ and consequently, because of Equation~\eqref{eq-omega1}, $\omega_1 \neq 0$. Then, comparing the coefficients of $\ba_1 \text{ and } \bb_1$, we obtain,
\begin{align}\label{eq-t1}
    \cos(\omega_1 (t + T) ) = \cos(\omega_1 t)\quad \text{ and }\quad \sin(\omega_1 (t + T) ) = \sin( \omega_1 t)
\end{align}
for each $t \in \mathbb{R}$. This shows that there exists a non-zero $m \in \mathbb{Z}$ such that,
    \[T = \frac{2\pi m}{\omega_1}.\]
Similarly, we compare the coefficients of $\ba_2 \text{ and } \bb_2$, to get
\begin{align}\label{eq-t2}
    \cos(\omega_2 (t + T) ) = \cos( \omega_2 t)\quad \text{ and }\quad \sin(\omega_2 (t + T) ) = \sin( \omega_2 t)
\end{align}
for each $t \in \rl$. This shows that there exists a non-zero $n \in \mathbb{Z}$ such that, 
    \[T = \frac{2\pi n}{\omega_2}.\]
It follows that
    \[\frac{\omega_1}{\omega_2} = \frac{m}{n} \in \mathbb{Q}\]
Now we prove the converse. Suppose that $\cfrac{\omega_1}{\omega_2} = \cfrac{m}{n} \in \mathbb{Q}$. Then, let
    \[T = \frac{2\pi m}{\omega_1} = \frac{2\pi n}{\omega_2}.\]
Then, Equations \eqref{eq-t1} and \eqref{eq-t2} hold. Therefore, 
    \[\bs{\gamma}(t + T) = \bs{\gamma}(t)\]
Which proves part 2 of Theorem~\ref{thm-existence} in the case where $\tau \neq 0$. If, on the other hand, $\tau = 0$ and consequently $\omega_1 = 0$, then $\bs{\gamma}$ is given by,
$$\bs{\gamma}(t) = \ba_1 + \cos(\omega_2 t) \ba_2 + \sin(\omega_2 t) \bb_2.$$
Which is always periodic with a period of 
$$T = \frac{2 \pi}{\omega_2}.$$
This proves part 2 of the theorem.

\subsection{Part 3}
Let $\kappa, \tau > 0$ and let $\mathbf{a}_1, \mathbf{b}_1, \mathbf{a}_2, \text{ and } \mathbf{b}_2$ be the orthonormal basis of $\rl^4$ consisting of the unit vectors along the coefficient vectors $\ba_1, \bb_1, \ba_2, \text{ and } \bb_2$ of $\bs{\gamma}$ as given in Equation~\eqref{eq-lissajous}. We denote the coordinates of a point $\mathbf{x} \in \rl^4$ by
$$\mathbf{x} = x_1 \mathbf{a}_1 + x_2 \mathbf{b}_1 + x_3 \mathbf{a}_2 + x_4 \mathbf{b}_2.$$
By parts 2 and 5 of Theorem~\ref{thm-main}, we have that $\bs{\gamma}$ is represented in these coordinates by $x_1 = \abs{\ba_1} \cos(\omega_1 t), x_2 = \abs{\ba_1} \sin(\omega_1 t), x_3 = \abs{\ba_2} \cos(\omega_2 t), \text{ and } x_4 = \abs{\ba_2} \sin(\omega_2 t)$. Consider the torus in $\rl^4$ given by
\begin{align}
    \mathbb{T}^2_{\bs{\gamma}} = \left\{\mathbf{x} \in \rl^4 : x_1^2 + x_2^2 = \abs{\ba_1}^2, x_3^2 + x_4^2 = \abs{\ba_2}^2 \right\}.
\end{align}
Clearly $\bs{\gamma}$ lies on $\mathbb{T}^2_{\bs{\gamma}}$. It is clear that $\mathbb{T}^2_{\bs{\gamma}}$ is a flat Clifford torus in $\rl^4$, and  is contained in $\sx^3$, since if $\mathbf{x}=(x_1,x_2,x_3,x_4)\in \mathbb{T}^2_{\bs{\gamma}}$, then
\begin{align*}
    \abs{\mathbf{x}}^2=x_1^2 + x_2^2+ x_3^2 + x_4^2 &= \abs{\ba_1}^2+\abs{\ba_2}^2\\
        &= \frac{1 - \omega_2^2}{\omega_1^2 - \omega_2^2} + \frac{1 - \omega_1^2}{\omega_2^2 - \omega_1^2}\\
        &=1,
\end{align*}
where the last equality follows by Equations~\eqref{eq-a1b1}~and~\eqref{eq-a2b2}.

\subsection{Part 4}
Note that the helix~$\bs{\gamma}$ is a solution of the differential equations on the torus~$\mathbb{T}^2_{\bs{\gamma}}$
$$\frac{d \theta_1}{dt} = \omega_1 \text{ and } \frac{d \theta_2}{d t} = \omega_2$$
where $\theta_1$ (resp. $\theta_2$) is an angular coordinate on the circle $x_1^2 + x_2^2 = \abs{\ba_1}^2$ (resp. $x_3^2 + x_4^2 = \abs{\ba_2}^2$). If $\cfrac{\omega_1}{\omega_2}\notin \mathbb{Q}$, then a classical result in the theory of dynamical systems \cite[Proposition 4.2.8, p. 113]{katok} shows that the image of $\bs{\gamma}$ is dense in $\mathbb{T}^2_{\bs{\gamma}}$. Therefore, part 3 of the theorem is proven.

\end{document}